\documentclass{amsart}
\usepackage{amsfonts,amssymb,amsmath,amsthm}
\usepackage{url}
\usepackage{enumerate}

\input xy
\xyoption{all}

\newcommand{\on}{\operatorname}
\newcommand{\SU}{\operatorname{SU}}
\newcommand{\SO}{\operatorname{SO}}

\newcommand{\hra}{\hookrightarrow}
\newcommand{\f}{\frac}
\newcommand{\Co}{\mathcal{C}}
\newcommand{\M}{\mathcal{M}}
\newcommand{\om}{\omega}
\newcommand{\C}{\mathbb{C}}
\newcommand{\Z}{\mathbb{Z}}
\newcommand{\R}{\mathbb{R}}
\newcommand{\N}{\mathbb{N}}
\newcommand{\g}{\mathfrak{g}}
\renewcommand{\t}{\mathfrak{t}}
\newcommand{\pr}{\mathrm{pr}}

\newcommand{\mat}[4]{(\begin{smallmatrix} #1 & #2 \\ #3 & #4\end{smallmatrix})}
\newcommand{\I}{{t_*}}
\newcommand{\J}{{u_*}}
\newcommand{\D}{\mathcal{D}}
\newcommand{\su}{\mathfrak{su}}
\newcommand{\Ad}{\mathrm{Ad}}

\newcommand{\wt}{\widetilde}
\newcommand{\ti}{\tilde}
\newcommand{\ca}{\mathcal}
\newcommand\dirac{/\kern-1.2ex\partial} 
\theoremstyle{plain}
\newtheorem{thm}{Theorem}[section]
\newtheorem{lemma}[thm]{Lemma}

\newtheorem{prop}[thm]{Proposition}

\theoremstyle{definition}
\newtheorem{defn}[thm]{Definition}

\theoremstyle{remark}
\newtheorem{remark}[thm]{Remark}

\title{On the Verlinde formulas for $\SO(3)$-bundles}

\author{Derek Krepski}
\address{
Department of Mathematics and Statistics\\
McMaster University\\
Hamilton, ON}
\email{Derek.Krepski@math.mcmaster.ca}

\author{Eckhard Meinrenken}
\address{
Department of Mathematics \\
University of Toronto \\
Toronto, ON}
\email{mein@math.toronto.edu}

\begin{document}
\maketitle


\begin{abstract}
This paper computes the quantization of the moduli space of flat $\SO(3)$-bundles over an oriented surface with boundary, with  prescribed holonomies around the boundary circles. The result agrees with the generalized Verlinde formula conjectured by Fuchs and Schweigert.  
\end{abstract}

\section{Introduction}
Let $G$ be a compact, connected Lie group, $\Sigma$ a compact oriented surface of genus $h$ with $r$ boundary components. 
Given conjugacy classes $\Co_1,\ldots,\Co_r\subset G$, denote by 
\begin{equation}\label{eq:modspace}
 \M(\Sigma,\Co_1,\ldots,\Co_r)
 \end{equation}
the moduli space of flat $G$-bundles over $\Sigma$, with 
boundary holonomies in prescribed conjugacy classes
$\Co_j$. The choice of an invariant inner product on $\g$ defines a symplectic structure 
on the moduli space. Under suitable integrality conditions the moduli space carries a 
pre-quantum line bundle $L$, and one can define the \emph{quantization}
\begin{equation}\label{eq:quant}
 \mathcal{Q}(\M(\Sigma,\Co_1,\ldots,\Co_r))\in\Z
\end{equation}
as the index of the $\on{Spin}_c$-Dirac operator with coefficients in
 $L$.  (It may be necessary to use a partial desingularization as in \cite{me:si}.) Choosing a complex structure on
$\Sigma$ further defines a K\"ahler structure on the moduli space. If $G$ is
simply connected, Kodaira vanishing results \cite{te:bo} show that the above
index coincides with the dimension of the space of holomorphic
sections of $L$. It is given by the celebrated Verlinde formula
\cite{ve:fr,ts:cf,fa:vl,sz:ver,be:cb}.  For symplectic approaches to
the Verlinde formulas, much in the spirit of the present paper, see
\cite{je:in,je:ve,je:bo,bi:sy,al:fi}.

Much less is known for non-simply connected groups. For surfaces without boundary ($r=0$), and  taking $G=\on{PU}(n)$, Verlinde-type formulas were obtained by Pantev \cite{pa:cm} in the case $n=2$ and by 
Beauville \cite{be:ve} for $n$ prime. For more general compact, semi-simple connected 
Lie groups, Fuchs and Schweigert \cite{fu:ac} conjectured a generalization of the 
Verlinde formula, expressed in terms of \emph{orbit Lie algebras}. Partial results on these conjectures were 
obtained in \cite{al:fi}. 

In this article, we will establish Fuchs-Schweigert formulas for the
index \eqref{eq:quant} for the simplest case $G=\SO(3)$. We will use
the recently developed quantization procedure \cite{me:su, me:twi} for
quasi-Hamiltonian actions with group-valued moment map \cite{al:mom}.  
In order to apply these techniques, we present the moduli spaces (\ref{eq:modspace}) as
symplectic quotients of quasi-Hamiltonian $\ti{G}$-spaces for the
universal cover $\ti{G}=\SU(2)$.  In more detail, let $\D_i\subset
\SU(2)$ be conjugacy classes, and consider the quasi-Hamiltonian
$\SU(2)$-space
\[ \ti{M}=\D_1\times \cdots \times \D_s\times \SU(2)^{2h}\]
with moment map the product of holonomies, 
\[ \ti{\Phi}(d_1,\ldots,d_s,a_1,b_1,\ldots,a_h,b_h)=\prod_{i=1}^s d_i \prod_{j=1}^h (a_jb_ja_j^{-1}b_j^{-1}).\]
Put $M=\ti{M}/\Gamma$, where $\Gamma\subset Z^{s+2h}$ is the subgroup
preserving $\ti{M}\subset \SU(2)^{s+2h}$ and $\ti{\Phi}$. Then $M$ is
a quasi-Hamiltonian $\SU(2)$-space, and all connected components of
moduli spaces \eqref{eq:quant} are symplectic quotients $M/ \!\!
/\SU(2)$ for suitable choices of $\D_j$ (see Section \ref{sec:modspace}). Our first main result gives
necessary and sufficient conditions under which the space $M$ admits a
level $k$ pre-quantization \cite{kre:pr}. Using localization, we then
compute the corresponding quantization $\ca{Q}(M)\in R_k(\SU(2))$, an
element of the level $k$ fusion ring (Verlinde ring).  These results are summarized in Theorem \ref{th:main}.  We reformulate
the result as an equivariant version of the Fuchs-Schweigert formula (Theorem \ref{th:FSconjecture});
the non-equivariant formula (see (\ref{eq:noneqFS}) in Section \ref{sec:Smatrix}) is
then obtained from a `quantization commutes with reduction' principle.

Using the results of \cite{me:twi}, it is also possible to compute
quantizations of moduli spaces for non-simply connected groups of higher
rank. However, the determination of the pre-quantization conditions
and the evaluation of the fixed point contributions becomes more
involved. We will return to these questions in a forthcoming paper;
see also the author's abstracts in Oberwolfach
Report No.~ 2011/09.

\section{Preliminaries}\label{sec:prel}
The following notation, consistent with \cite{me:su}, will be used in this paper.  For the Lie group $\SU(2)$ let $T$ be the maximal torus given as the image of 
$$j\colon \on{U}(1)\to \on{SU}(2),\ \ \ 
j(z) = \mat{z}{0}{0}{\bar{z}}.
$$ Let $\Lambda = \ker \exp_T \subset \mathfrak{t}$ denote the
 integral lattice and $\Lambda^*\subset \mathfrak{t}^*$ its dual, the
 (real) weight lattice. 
Let $\rho \in \Lambda^*$ be the generator dual to the generator $\mathrm{d}
 j(2\pi i) \in \Lambda$.  We will use the \emph{basic inner product} on
 $\su(2)$,
$$ \xi \cdot \xi': = \frac{1}{4\pi^2} \mathrm{tr}(\xi^\dagger \xi'),
\quad\quad \xi,\xi' \in \su(2)
$$
to identify $\su(2) \cong \su(2)^*$.  Under this identification, $||\rho||^2=\f{1}{2}$, 
and $\Lambda = 2 \Lambda^*$ with generator $2\rho$. The following two elements of $\SU(2)$ 
will play a special role in this paper: 
\[ \J=\mat{0}{1}{-1}{0}, \ \ \I=\mat{i}{0}{0}{-i}\] 
Observe that $\I=\exp(\rho/2)$, with square $c=\exp \rho$ the non-trivial element in the center $Z:=Z(\SU(2))\cong \Z_2$. The element $\J\in N(T)$ represents the non-trivial element of the Weyl
group $W=N(T)/T\cong \Z_2$. Both $\J,\I$ are contained in the conjugacy class $\D_*\subset \SU(2)$ of 
elements of trace $0$. Note that $\D_*$ is the unique conjugacy class in $\SU(2)$ that is invariant 
under multiplication by $Z$. The quotient $\Co_*=\D_*/Z\cong \R P(2)$ is the conjugacy class in $\SO(3)$ 
consisting of rotations by $\pi$.

\subsection{The fusion ring $R_k(\SU(2))$}
We view the representation ring $R(\SU(2))$ as the subring of
$C^\infty(\SU(2))$ generated by characters of
$\SU(2)$-representations. As a $\Z$-module, it is free with basis
$\chi_0,\,\chi_1,\,\chi_2,\ldots $, where $\chi_m$ is the character of
the irreducible $\SU(2)$-representation on the $m$-th symmetric power $S^m(\C^2)$. The ring
structure is determined by the formula
$$
\chi_m \chi_{m'} = \chi_{m+m'} + \chi_{m+m'-2} + \cdots +\chi_{|m-m'|}.
$$
For $k=0,1,2,\ldots$ let $I_k(\SU(2))$ be the ideal generated by $\chi_{k+1}$ and let
$$ R_k(\SU(2)) = R(\SU(2))/I_k(\SU(2)) $$
be the \emph{level $k$ fusion ring} (or \emph{Verlinde ring}).  As a $\Z$-module, $R_k(\SU(2))$ is free, with basis $\tau_0, \tau_1, \ldots , \tau_k$ the images of $\chi_0, \chi_1, \ldots, \chi_k$ under the quotient homomorphism. Let $q=e^{\frac{i\pi}{k+2}}$ be the $2k+4$-th root of unity, and define \emph{special points}
\begin{equation}\label{eq:special} t_l=j(q^{l+1}),\ \ l=0,\ldots,k.\end{equation}
Then $I_k(\SU(2)) \subset R(\SU(2))$ has an alternative description as the ideal of 
characters vanishing at all special points \eqref{eq:special}. Hence, the evaluation of 
characters at the special points descends to evaluations 
$R_k(\SU(2))\to \C,\ \tau\mapsto \tau(t_l)$. 

The product in the complexified fusion ring $R_k(\SU(2))\otimes_\Z \C$ can be diagonalized 
using the \emph{$S$-matrix}, given by the Kac-Peterson formula 
\begin{equation} \label{Kac-Peterson}
S_{m,l} = (\tfrac{k}{2}+1)^{-\f{1}{2}} \sin\big( \tfrac{\pi(l+1)(m+1)}{k+2}\big),
\end{equation}
for $l,m= 0, 1, \ldots, k$. The $S$-matrix is orthogonal, and the alternative basis elements 
\[ \ti{\tau}_l = \sum_m S_{0,l} S_{m,l} \tau_m\]
satisfy $\ti{\tau}_m(t_l)=\delta_{m,l}$, hence 
\[ \ti{\tau}_{m}\ti{\tau}_{m'}=\delta_{m,m'}\ti{\tau}_{m}.\] 
The basis elements $\{ \tau_0, \ldots, \tau_k\}$ are expressed in terms of the alternative basis as 
$\tau_m=\sum_{l} S_{0,l}^{-1} S_{m,l}\ti{\tau}_l$. 

\subsection{Quasi-Hamiltonian $G$-spaces} \label{sec:quasi}
We recall some basic definitions and facts from \cite{al:mom}. Let $G$ be a compact Lie group with Lie algebra $\mathfrak{g}$, equipped with an invariant inner product, denoted by a dot $\cdot$.  Let $\theta^L,\,\theta^R$ denote the left-invariant, right-invariant Maurer-Cartan forms on $G$, and let 
$\eta= \tfrac{1}{12} \theta^L \cdot [\theta^L,\theta^L]$ denote the Cartan 3-form on $G$.  For a $G$-manifold $M$, and $\xi \in \mathfrak{g}$, let $\xi^\sharp$ denote the generating vector field, 
defined in terms of the action on functions $f\in C^\infty(M)$ by 
$(\xi^\sharp f)(x) = \frac{d}{dt}\Big|_{t=0} f(\exp (-t\xi). x).
$
The Lie group $G$ is itself viewed as a $G$-manifold for the conjugation action. 

\begin{defn}
A quasi-Hamiltonian $G$-space is a triple $(M,\omega,\Phi)$ consisting of a $G$-manifold $M$, a $G$-invariant 2-form $\omega$ on $M$, and an equivariant map $\Phi\colon M \to G$, called the moment map, satisfying:
\begin{enumerate}
\item $d\omega + \Phi^* \eta = 0 $,
\item $\iota_{\xi^\sharp}\omega +\tfrac{1}{2} \Phi^*((\theta^L +
\theta^R) \cdot \xi )=0$ for all $\xi \in \mathfrak{g}$,
\item at every point $x\in M$, $\ker \omega_x \cap \ker \mathrm{d} \Phi_x = \{0\}$.
\end{enumerate}
\end{defn}

The \emph{fusion product} of two quasi-Hamiltonian $G$-spaces $(M_1,\omega_1,\Phi_1)$ and $(M_2,\omega_2,\Phi_2)$
is the product $M_1\times M_2$, with the diagonal $G$-action, 2-form 
\begin{align} \label{fusionform}
\omega &= \pr_1^*\omega_1 + \pr_2^*\omega_2 + \tfrac{1}{2} \pr_1^*\Phi_1^*\theta^L \cdot \pr_2^*\Phi_2^*\theta^R,
\end{align}
 and moment map $\Phi=\Phi_1\Phi_2$.
 
 The \emph{symplectic quotient} of a quasi-Hamiltonian $G$-space is the symplectic space $M/\!/G = \Phi^{-1}(e)/G$.  Similar to the theory of Hamiltonian group actions, the group unit $e$ is a regular value of $\Phi$ if and only if $G$ acts locally freely on the level set $\Phi^{-1}(e)$, and in this case the pull-back of the 2-form to the 
 level set descends to a symplectic 2-form on the orbifold $\Phi^{-1}(e)/G$. 
If $e$ is a singular value, then $M/\!/G$ is a singular symplectic space as defined in \cite{sj:st}.

The conjugacy classes $\mathcal{C}\subset G$ are basic examples of
quasi-Hamiltonian $G$-spaces.  The moment map is the inclusion into $G$, and
the 2-form $\omega$ is given on generating vector fields by the
formula
\begin{equation}\label{eq:conjclass}
\omega_g(\zeta^\sharp(g), \xi^\sharp(g)) = \tfrac{1}{2}( \xi \cdot
\Ad_g \zeta - \zeta \cdot \Ad_g \xi).\end{equation} 
Together with the
\emph{double} $\mathbf{D}(G)=G\times G$, equipped with diagonal
$G$-action and moment map $\Phi(g,h)=ghg^{-1}h^{-1}$, these are the
building blocks of the main example appearing in this paper.  As shown
in \cite{al:mom}, the moduli space of flat $G$-bundles over a compact,
oriented surface $\Sigma$ of genus $h$ with $s$ boundary components,
with boundary holonomies in prescribed conjugacy classes
$\mathcal{C}_j$, $j=1, \ldots, s$, is a symplectic quotient of a
fusion product:
\begin{equation} \label{eq:moduli}
M(\Sigma, \mathcal{C}_1, \ldots , \mathcal{C}_s) =\mathcal{C}_1 \times \cdots \times \mathcal{C}_s \times \mathbf{D}(G)^h /\!/G.
\end{equation}

If the group $G$ is simply connected, then the fibers of the moment
map for any compact, connected quasi-Hamiltonian $G$-space are
connected. In particular, \eqref{eq:moduli} is connected in that case.
If $G$ is non-simply connected, the space \eqref{eq:moduli} may have
several components.

To clarify the decomposition into components, we use the following
construction.  Suppose $p\colon \check{G}\to G$ is a homomorphism of
compact, connected Lie groups, with finite kernel $Z$. Then $Z$ is a
subgroup of the center of $\check{G}$, and $G=\check{G}/Z$.  For any
quasi-Hamiltonian $G$-space $(N,\omega,\Phi)$, let $\check{N}$ denote
the fiber product defined by the pull-back square
\begin{equation} \label{pbdiagram}
\xymatrix{
\check{N} \ar[r]^{\check{\Phi}} \ar[d]^{{p}_N} & \check{G} \ar[d]^{p} \\
N \ar[r]^{\Phi} & G
}
\end{equation}
Then $(\check{N}, \check{\omega}, {\check{\Phi}})$ is a quasi-Hamiltonian $\check{G}$-space, 
for the diagonal $\check{G}$-action on $\check{N}\subset N\times \check{G}$, and with the 2-form 
$\check{\omega} = {p}_N^*\omega$. 
Simple properties of this construction are:

\begin{prop}\label{pb} 
\begin{enumerate}
\item[(i)]
We have a canonical identification of symplectic quotients
\[ \check{N}/\!/\check{G} \cong N/\!/ G.\] 
\item[(ii)] For a fusion product $N={N}_1\times \cdots \times {N}_r$ 
of quasi-Hamiltonian ${G}$-spaces, the space $\check{N}$ is a 
quotient of $\check{N}_1\times\cdots \times \check{N}_r$ by the group  
$\{(c_1,\ldots,c_r)\in Z^r|\ \prod_{j=1}^r c_j=e\}$.
\item[(iii)] If $\Phi\colon N\to G$ lifts to a moment map $\Phi'\colon N\to \check{G}$, thus turning $N$ into a quasi-Hamiltonian $\check{G}$-space
$(N,\omega,\Phi')$, then 
\[ \check{N}=N\times Z\]
as a fusion product of quasi-Hamiltonian $\check{G}$-spaces. Here $Z$
 is viewed as a quasi-Hamiltonian  $\check{G}$-space, 
with trivial action and with moment map the inclusion to $\check{G}$.  
\end{enumerate}
\end{prop}

\begin{proof} 
(i) By definition of $\check{N}$, the level sets ${\check{\Phi}}^{-1}(\check{e})$ and $\Phi^{-1}(e)$ are identified, and the pull-backs of the 2-forms to the level sets coincide. Since central elements in $\check{G}$ act trivially on $\check{N}$, the orbit spaces ${\check{\Phi}}^{-1}(\check{e})/\check{G}$ and $\Phi^{-1}(e)/G$ are identified as well.  

(ii) Think of the spaces $\check{N}_i$ as submanifolds of $N_i\times \check{G}$. The canonical map 
\[ \check{N}_1 \times \cdots \times \check{N}_r\to \check{N},\ 
(x_1,g_1,x_2,g_2,\ldots,x_r,g_r)\mapsto (x_1,\ldots,x_r,g_1,\ldots,g_r)\]
is exactly the quotient map by $\{(c_1,\ldots,c_r)\in Z^r|\ \prod_{j=1}^r c_j=e\}$, and it preserves the $\check{G}$-actions and 2-forms. 

(iii) The map $N\times Z\to \check{N},\ (x,c)\mapsto 
(x,\Phi'(x)c)$ is the desired diffeomorphism.
\end{proof}

\subsection{The moduli space example} \label{sec:modspace}
Our main interest is the moduli space of flat $\SO(3)$-bundles with prescribed boundary 
holonomies, i.e.~ \eqref{eq:moduli} with $G=\SO(3)$. In the notation of the previous 
Section, we will describe the 
quasi-Hamiltonian $\SU(2)$-space $\check{N}$ associated to the quasi-Hamiltonian $\SO(3)$-space 
\[ N=\Co_1\times\cdots \times \Co_s \times \mathbf{D}(\SO(3))^h.\] 
Choose conjugacy classes $\D_j\in \SU(2)$ with $p(\D_j)=\Co_j$, and define a 
quasi-Hamiltonian $\SU(2)$-space 
\begin{equation}\label{eq:tim}
\ti{M}=\D_1\times\cdots \times \D_s\times \mathbf{D}(\SU(2))^h.
\end{equation}
Put 
\begin{equation}\label{eq:M}
 M=\ti{M}/\Gamma
,\end{equation}
where $\Gamma\subset Z^{s+2h}$ consists of  $\gamma=(\gamma_1,\ldots,\gamma_{s+2h})$
with the properties $\prod_{j=1}^s\gamma_j=e$ and $\gamma_j\D_j=\D_j$ for $ j\le s$.
(Equivalently, $\gamma_j=e$ for all $\D_j\neq\D_*$). The conditions guarantee that $\gamma$ acts on $\ti{M}$,
preserving the 2-form and moment map which hence descend to $M=\ti{M}/\Gamma$.
Let $\Co_*\cong \R P(2)$ be the $\SO(3)$-conjugacy class consisting of rotations
by $\pi$. It is the unique $\SO(3)$-conjugacy class whose pre-image in $\SU(2)$ 
is connected. This pre-image is the $\SU(2)$-conjugacy class $\D_*\cong S^2$  of matrices of trace $0$.

\begin{lemma}\label{lem:check} 
With $N$  as above, we have 
\[ \check{N}\cong \begin{cases} M &\mbox{ if }\ \ \exists\  j\colon \Co_j=\Co_*\\
M\times Z  &\mbox{ if }\ \ \forall\  j\colon \Co_j\neq\Co_*.
\end{cases}\] 
\end{lemma}
\begin{proof}
The moment map $\mathbf{D}(\SU(2))\to \SU(2)$ (given by Lie group
commutator) is invariant under the action of $Z\times Z$, hence it
descends to a lift $\mathbf{D}(\SO(3))\to \SU(2)$ of the 
commutator map for $\SO(3)$. Thus
\[ \check{\mathbf{D}}(\SO(3))=\mathbf{D}(\SO(3))\times Z.\]
If $\Co_j\neq \Co_*$, the map $\D_j\to \Co_j$ is a diffeomorphism, and defines 
a lift of the moment map $\Co_j\hra \SO(3)$. Hence 
\[ \check{\Co}_j= \D_j\times Z\]
in that case. On the other hand, the conjugacy class $\Co_*$ satisfies 
\[ \check{\Co}_*=\D_*.\]
With these ingredients, the claim follows from Proposition \ref{pb}. 
\end{proof}

We may choose the labeling of the conjugacy classes $\Co_1,\ldots,\Co_s$ in such a way that 
$\Co_j=\Co_*$ for $j \le r$ and $\Co_j\neq \Co_*$ for $j>r$. 
The space \eqref{eq:M} is then a fusion product 
\begin{equation}\label{eq:dirprod}
 M=M'\times \D_{r+1}\times \cdots \times \D_s\times \mathbf{D}(\SO(3))^h,
\end{equation}
where $\mathbf{D}(\SO(3))$ is viewed as a quasi-Hamiltonian $\SU(2)$-space 
(using the canonical lift of the $\SO(3)$ moment map, as in the proof of Lemma 
\ref{lem:check}), and where 
\[ M'=(\D_*\times \cdots \times \D_*)/\Gamma'\]
with $r$ factors, and with $\Gamma'=\{(\gamma_1,\ldots,\gamma_r)\in Z^r|\ \prod \gamma_j=e\}$.  Let us describe the 2-form $\om'$ of the space $M'$, in terms of its 
pull-back $\ti{\om}'$ to the universal cover $\tilde{M}=\D_*\times\cdots \times \D_*$. 
Since the 2-form on $\D_*$ is just zero, only the fusion terms contribute. 
By iterative use of the formula
(\ref{fusionform}) for the fusion product, one obtains
\begin{equation}\label{eq:hatomega}
\ti{\omega}'=\tfrac{1}{2} \sum_{i<j} g_i^*\theta^L\cdot \Ad_{g_{i+1}\cdots g_{j-1}}(g_j^*\theta^R),
\end{equation}
where $g_i\colon \ti{M}\to \D_*\subset \SU(2)$ denotes projection onto the $i$-th factor. 

\section{Quantization of the moduli space of flat $\SO(3)$-bundles}
In this section we use localization to compute the quantization of the space $M=(\D_1\times\cdots\times\D_s\times \mathbf{D}(\SU(2))^h)/\Gamma$, as an element of the level $k$ fusion ring $R_k(\SU(2))$.

\subsection{Pre-quantization}\label{subsec:preq}
Recall that we fix the inner product $\cdot$ on $\su(2)$ to be the
\emph{basic inner product}. Then $\eta\in \Omega^3(\SU(2))$ is integral, and
represents a generator $x\in H^3(\SU(2);\Z)\cong \Z$.  The condition
$d\omega+\Phi^*\eta=0$ from the definition of a quasi-Hamiltonian
space says that the pair $(\omega,\eta)$ defines a relative cocycle in
$\Omega^3(\Phi)$, the algebraic mapping cone of the pull-back map
$\Phi^*\colon \Omega^*(G)\to\Omega^*(M)$. Let $k\in \N$.
\begin{defn} \label{preq} \cite{kre:pr, me:su}
A \emph{level $k$ pre-quantization} of a quasi-Hamiltonian
$\SU(2)$-space $(M,\omega,\Phi)$ is an integral lift $\alpha\in
H^3(\Phi\, ;\Z)$ of the class $k[(\omega,\eta)]\in H^3(\Phi\,
;\R)$. \end{defn}
A necessary and sufficient condition for the existence of a level $k$ pre-quantization is that for all smooth singular 2-cycles $\Sigma\in Z_2(M)$, and all 
smooth singular 3-chains 
$C\in C_3(G)$ such that $\partial C=\Phi(\Sigma)$, 
\[ k\big(\int_\Sigma\om +\int_C \eta\big)\in\Z.\]
We list some basic properties and examples of level $k$
pre-quantizations.
\begin{enumerate}
\item[(a)] 
The set of level $k$ pre-quantizations is a torsor under the torsion group $\on{Tor}(H^2(M,\Z))$ of isomorphism classes of flat line bundles. 
\item[(b)] The level $k$ pre-quantized conjugacy classes of $\SU(2)$
are exactly those of the elements $\exp(\frac{m}{k}\rho)$ with
$m=0,\ldots,k$ \cite[Proposition 7.3]{me:su}.
\item[(c)] The double $\mathbf{D}(\SO(3))$ (viewed as a quasi-Hamiltonian $\SU(2)$-space) admits a level $k$ pre-quantization if and only if $k$ is even \cite[Proposition 7.4]{me:su}. 
\item[(d)] 
If $M_1$ and $M_2$ are pre-quantized quasi-Hamiltonian $\SU(2)$-spaces at level $k$, then their fusion product $M_1 \times M_2$ inherits a pre-quantization at level $k$.  Conversely, a pre-quantization of the product induces pre-quantizations of the factors. See \cite[Proposition 3.8]{kre:pr}. 
 \item[(e)] A level $k$ pre-quantization 
of $M$ induces a pre-quantization of the symplectic quotient $M/ \!/ \SU(2)$, equipped with the $k$-th multiple of the symplectic form. 
\item[(f)] 
The long exact sequence in relative cohomology gives a necessary condition 
$k\Phi^*(x)=0$ for the existence of a level $k$ pre-quantization. If $H^2(M;\R)=0$, this condition is also sufficient \cite[Proposition 4.2]{kre:pr}.  
\item[(g)] 
The existence of the canonical `twisted $\on{Spin}_c$-structure' \cite[Section 6]{me:su} on quasi-Hamiltonian $\SU(2)$-spaces $(M,\omega,\Phi)$ implies that 
$2\Phi^*(x)=W^3(M)$, the third integral Stiefel-Whitney class. Since this is a 2-torsion class, $4\Phi^*(x)=0$. In fact, there is a distinguished element $\beta\in H^3(\Phi\,;\Z)$ whose image in $H^3(\SU(2)\,;\Z)$ is $4x$. If $H^2(M,\R)=0$, this element gives a distinguished level $4$ pre-quantization. 
\end{enumerate}

Given a level $k$ pre-quantization of a quasi-Hamiltonian $\SU(2)$-space $(M,\om,\Phi)$ 
the construction from \cite{me:su} produces a \emph{quantization} $\ca{Q}(M)\in R_k(\SU(2))$, 
an element of the level $k$ fusion ring. It is obtained as a push-forward 
in twisted equivariant $K$-homology, using the Freed-Hopkins-Teleman theorem \cite{fr:lo1} 
to identify $R_k(\SU(2))$ with the equivariant twisted $K$-homology of $\SU(2)$ at level 
$k+2$. This is the
quasi-Hamiltonian counterpart of the $\mathrm{Spin}_c$ quantization of
an ordinary compact Hamiltonian $\SU(2)$-space, which produces an element of 
$R(\SU(2))$ as the equivariant index of a
$\mathrm{Spin}_c$-Dirac operator with coefficients in an equivariant
pre-quantum line bundle. The quantization procedure for
quasi-Hamiltonian $G$-spaces satisfies properties similar to its
Hamiltonian analog.  These include
\begin{enumerate}
\item compatibility with products, $\mathcal{Q}(M_1 \times M_2) = \mathcal{Q}(M_1)\mathcal{Q}(M_2)$; and
\item the `quantization commutes with reduction' principle, 
  $\mathcal{Q}(M/\!\!/ G)=\ca{Q}(M)^G $. 
  \end{enumerate}
Here $R_k(G)\to \Z,\ \tau\mapsto \tau^G$ is the trace defined by
$\tau_m^G=\delta_{m,0}$.

\subsection{Pre-quantization of $M$} 
Let us now consider level $k$ pre-quantizations of the quasi-Hamiltonian $\SU(2)$-space 
\[ M=(\D_1\times \cdots \times \D_s\times \mathbf{D}(\SU(2))^h)/\Gamma\] 
from \eqref{eq:M}. 
\begin{thm}\label{th:preq}
The quasi-Hamiltonian $\SU(2)$-space $M$ carries a level $k$ pre-quantization if and only if the following conditions are satisfied:
\begin{enumerate}
\item[(i)] The conjugacy classes $\D_j$ are of the form $\SU(2).\exp(\f{m_j}{k}\rho)$ with
$m_j\in\{0,\ldots,k\}$, 
\item[(ii)] if $h\ge 1$, then $k\in 2\N$, 
\item[(iii)] if the number of $\D_*$-factors is $r\ge 3$, then $k\in 4\N$. 
\end{enumerate}
\end{thm}
Note that if at least one $\D_*$-factor appears, then the first condition requires that $k\in 2\N$
since $\D_*=\SU(2).\exp(\f{1}{2}\rho)$.  
\begin{proof}
Since a level $k$ pre-quantization of $M$ induces a level $k$ pre-quantization 
of the universal cover $\ti{M}$, it is a necessary condition that all $\D_j$ be pre-quantizable.  
That is, $\D_j=\SU(2).\exp(\f{m_j}{k}\rho)$ with $m_j\in\{0,\ldots,k\}$.  

Let us enumerate the conjugacy classes in such a way that
$\D_1=\ldots=\D_r=\D_*$.  Using the decomposition \eqref{eq:dirprod}
and the known pre-quantization conditions (b),(c) for the conjugacy classes 
$\D_j$ and the double $\mathbf{D}(\SO(3))$, 
together with the fusion property (d), the proof is reduced to the case 
$h=0,\ s=r$.
We may thus assume $M=(\D_*\times\cdots \times \D_*)/\Gamma$ with $r$
factors. If $r=1$ then $M=\D_*$, which is pre-quantized at level $k$
if and only if $k$ is even. Suppose $r>1$.  The non-trivial
element $c\in Z$ acts on $H^2(\mathcal{D}_*;\R)\cong\R$ as
multiplication by $-1$.  Hence, $\Gamma$ acts on $H^2(M;\R)\cong \R^r$
by componentwise sign changes. In particular, the $\Gamma$-invariant
part is trivial. Since $\Gamma$ acts freely, it follows that
\[ H^2(M\,;\R)\cong H^2(\ti{M}\,;\R)^\Gamma=0.\]
Hence, by Property (f),  a level $k$ pre-quantization exists if and only if 
$k \Phi^*(x)=0$. If $r=2$, so that
$M=(\D_*\times\D_*)/\Z_2$, Poincar\'e duality gives that
$H^3(M;\Z) \cong \Z_2$; therefore $2\Phi^*(x)=0$.
Hence the condition $k\in 2\N$ is also sufficient if $r=2$.

It remains to consider the case $r\ge 3$. By Property (g), the
condition $k\in 4\N$ is sufficient. Let us show that it is also
necessary. Observe that the non-identity component of the normalizer,
the circle $T\J=N(T)-T$, is a single conjugacy class inside
$N(T)$. Since $\J\in \D_*$, it follows that $T\J\subset \D_*$.  Let
$\ti{X}\subset \ti{M}=\D_*\times\cdots\times\D_*$ be the 2-torus given as the image of the map
\[ T\times T\to \ti{M},\ \ (h_1,h_2)\mapsto \big(h_1\J,h_2\J,h_1h_2\J,\J,\ldots,\J\big),\]
and denote by $X$ its image in $M$. Let $\ti{\om}_X,\om_X$ be the pull-backs of the 
quasi-Hamiltonian 2-forms on $\ti{X},\,X$. Since $T\J=\J T$, we have $\ti{\Phi}(\ti{X})=\Phi(X)
\subset T\J^r$. Since the generator $x\in H^3(\SU(2),\Z)$ pulls 
back to zero on this circle (for dimension reasons), the existence of a level $k$ pre-quantization of 
$M$ requires that $k\int_{X}\om_X\in\Z$. Since the projection $\ti{X}\to X$ 
is a 4-fold covering, $\int_{X}\om_X=\frac{1}{4}\int_{\ti{X}}\ti{\om}_X$. Hence it is necessary that 
$k\int_{\ti{X}}\ti{\om}_X\in 4\Z$. 

Let $\theta\in \Omega^1(T,\t)$ be the Maurer-Cartan form for $T$. 
From the general formula \eqref{eq:hatomega}, and using $(h\J)^*\theta^L=-h^*\theta,\ 
(h\J)^*\theta^R=h^*\theta$, 
we obtain 
\[ \ti{\om}_X=\frac{1}{2} \big(-h_1^*\theta\wedge h_2^*\theta+h_1^*\theta\wedge (h_1h_2)^*\theta
-h_2^*\theta\wedge (h_1h_2)^*\theta\big)=\frac{1}{2} h_1^*\theta\wedge h_2^*\theta.\]
Writing elements of $T$ in the form $h=j(e^{2\pi i v})$, we may take $v\in [0,1]$ 
as the coordinate on $T\cong \R/\Z$. Since the lattice $\Lambda$ is generated by $2\rho$, we find  
$h_i^*\theta=2 d v_i\otimes \rho$, hence
\[ \ti{\om}_X= 2||\rho||^2 \ d v_1\wedge  d v_2= d v_1\wedge d v_2\]
integrates to $1$. This gives the condition $k\in 4\N$. 
\end{proof}

\subsection{Fixed point components}
Suppose $M$ is a level $k$ pre-quantized quasi-Hamiltonian $\SU(2)$-space, and let 
$\ca{Q}(M)\in R_k(\SU(2))$ 
be its quantization. By \cite[Theorem 9.5]{me:su}, the numbers 
$\mathcal{Q}(M)(t)$ with $t=t_l,\ l=0,\ldots,k$ are given as a sum of contributions 
from the fixed point manifolds of $t$:
\begin{align} \label{localization}
\mathcal{Q}(M)(t) &= \sum_{F\subset M^t} \int_F \frac{\widehat{A}(F)\
\on{Ch}(\ca{L}_F,\, t)^{1/2}}{D_\R(\nu_F,t) }.
\end{align}
The ingredients of the right hand side will be described below, and explicitly computed in the context of our main example \eqref{eq:M}. The quantizations of $\SU(2)$-conjugacy classes and of the double 
$\mathbf{D}(\SO(3))$ (viewed as a quasi-Hamiltonian $\SU(2)$-space) were computed in \cite{me:su}. 

For the remainder of this section, we therefore focus on the case $h=0,\,s=r\ge 2$, i.e.~  
$M=(\D_*\times \cdots\times \D_*)/\Gamma$. 

\subsubsection{Fixed point sets of $M$}\label{subsec:fixedpointsets}
We need to determine the components 
$F\subset M^t$ of the 
fixed point manifold for $t=t_l$, $l=0, \ldots, k$, and describe various aspects of $F$ and its normal bundle $\nu_F$. Consider first a general regular element $t\in T^{\mathrm{reg}}$.  
Define the following two submanifolds of $\D_*$, labeled by the elements of the center $Z=\{e,c\}$
as follows: 
\[ Y^{(e)}=\D_*\cap T=\{\I,\I^{-1}\},\ \ Y^{(c)}=T\J.\]
Thus $Y^{(e)}$ is the fixed point set of $\Ad(\I)$, while $Y^{(c)}$ consists of elements satisfying 
$\Ad(\I)(g)=cg$. Note that both are $Z$-invariant.
%
For $\gamma=(\gamma_1,\ldots,\gamma_r) \in \Gamma$, consider the $\Gamma$-invariant submanifold 
$$
\ti{F}^{(\gamma)} = Y^{(\gamma_1)}\times\cdots \times Y^{(\gamma_r)}.
$$
and put $F^{(\gamma)} = \ti{F}^{(\gamma)}/\Gamma$. Let $\mathsf{l}(\gamma)$ be the number of 
$\gamma_i$'s that are equal to $c$. Then $\ti{F}^{(\gamma)}$ is a disjoint union of $2^{r-\mathsf{l}(\gamma)}$ tori 
of dimension $\mathsf{l}(\gamma)$. 
Let $\varepsilon=(e, \ldots, e)$ denote the group unit in $\Gamma$. If $\gamma\not=\varepsilon$, then $\Gamma$ acts transitively 
on the set of components of $\ti{F}^{(\gamma)}$. Hence $F^{(\gamma)}$ is a (connected) torus, and 
since $|\Gamma|=2^{r-1}$, it follows that the projection restricts 
to a $2^{\mathsf{l}(\gamma)-1}$-fold covering on each component of $\ti{F}^{(\gamma)}$. If $\gamma=\varepsilon$, 
$\ti{F}^{(\varepsilon)}$ consists of $2^r$ points, and hence $F^{(\varepsilon)}$ consists of two points. 

\begin{prop} \label{fixedpointsets}The fixed point set of $t\in T^{\mathrm{reg}}$ in $M$ is 
\[ M^t =\begin{cases}F^{(\varepsilon)} &\mbox{ if }\ \ t\notin \{\I, \I^{-1}\},\\
\coprod_{\gamma \in \Gamma} F^{(\gamma)}&\mbox{ if }\ \ t\in \{\I, \I^{-1}\}.
\end{cases}\]
\end{prop}
\begin{proof} 
An element $(g_1,\ldots,g_r)\in \ti{M}$ maps 
to a point in $M^t$ if and only if there exists $\gamma=(\gamma_1,\ldots,\gamma_r)\in\Gamma$ with 
$\Ad(t)g_i=g_i\gamma_i$, for $i=1,\ldots,r$. If $\gamma_i=e$, this condition gives 
$g_i\in T$, since $t$ is regular. If $\gamma_i=c$, the condition says that $\Ad(g_i^{-1})(t)=\gamma_i t$. 
Since $t$ is regular, this happens if and only if $t\in \{\I,\,\I^{-1}\}$, with $g_i\in N(T)$ representing 
the non-trivial Weyl group element. 
\end{proof}

\subsubsection{The symplectic volume of the components of the fixed point set} \label{sympvol}
Each $F^{(\gamma)}\subset M^t$ is a quasi-Hamiltonian $T$-space, with
moment map the restriction of $\Phi$.  (See e.g.~ \cite[Proposition
3.1]{me:twi}.)  In particular, they are symplectic.
\begin{lemma} \label{restofform}
The symplectic volume of  each component of $\ti{F}^{(\gamma)}$ is equal to $1$. Thus
\[ \on{vol}(F^{(\gamma)})=2^{1-\mathsf{l}(\gamma)}.\]
\end{lemma}
\begin{proof}
The construction from \cite{al:du} associates to any quasi-Hamiltonian $G$-space 
(with $G$ compact, but possibly disconnected) a Liouville volume, in such a way that the volume of a fusion product is the product of the volumes. If $G=T$, so that the space is symplectic, the Liouville volume coincides with the symplectic volume. For a $G$-conjugacy class $\Co\cong G/G_g$, the Liouville volume is given by the formula 
\cite[Proposition 3.6]{al:du}
\[ \on{vol}{\Co}=|{\det}_{\g_g^\perp}(1-\Ad_g)|^{1/2} \frac{\on{vol}(G)}{\on{vol}(G_g)},\] 
involving the Riemannian volumes of $G$ and of the stabilizer group $G_g$. 
The spaces $Y^{(z)}$ for $z\in Z$ can be viewed as conjugacy classes for the group $N(T)$, of elements 
$\I$ if $z=e$ and $\J$ if $z=c$. Application of the formula gives 
\[ \on{vol}(Y^{(z)})=\begin{cases} 2 & \mbox{ if } z=e\\ 1& \mbox{ if } z=c\end{cases}.\]
This is obvious for $z=e$, while for $z=c$ (so that $g=\J,\ N(T)_g=\Z_4$) we have  
$|\det_{\t}(1-\Ad_\J)|^{1/2}=\sqrt{2}$ (since $\Ad_{\J}$ acts as $-1$ on $\t$), 
$\on{vol}(N(T))=2\on{vol}(T)=2 ||\alpha||=2\sqrt{2}$, and $\on{vol}(N(T)_g)=4$. 
It follows that 
\[ \on{vol}(\ti{F}^{(\gamma)})=\prod_{i=1}^r \on{vol}(Y^{(\gamma_i)})=2^{r-\mathsf{l}(\gamma)}.\]  
Since the moment map for the quasi-Hamiltonian $N(T)$-space $\ti{F}^{(\gamma)}$ takes values in $T$, 
this coincides with the symplectic volume. Since $2^{r-\mathsf{l}(\gamma)}$ is also the number of components 
of $\ti{F}^{(\gamma)}$, it follows that each component has volume $1$. 
\end{proof}

\subsection{Fixed point contributions}
In this Section, we assume that $M=(\D_*\times\cdots\times\D_*)/\Gamma$ carries a level $k$ pre-quantization. Thus $k\in 2\N$ if $r=2$ 
and $k\in 4\N$ if $r>2$. 
Our aim is to compute the fixed point contributions to $\mathcal{Q}(M)(t)$, as described in 
formula (\ref{localization}), for $t=t_l$, $l=0, \ldots, k$. 

If $t\neq \I$, Proposition \ref{fixedpointsets} shows that $M^t=F^{(\varepsilon)}$ consists of 
just two points, covered by the set $\ti{M}^t=\ti{F}^{(\varepsilon)}$ (consisting of 
$2^r$ points). The fixed point contribution of $F^{(\varepsilon)}$ is just that for $\ti{F}^{(\varepsilon)}$, divided 
by $|\Gamma|=2^{r-1}$. Hence 
\[ \mathcal{Q}(M)(t)=2^{1-r}\mathcal{Q}(\ti{M}^t)=2^{1-r}\mathcal{Q}({\mathcal{D}_*})^r(t),\]
with $\ca{Q}(\D_*)=\tau_{k/2}$ \cite[Proposition 11.2]{me:su} . 

If $t=\I$, $\ca{Q}(M)(\I)$ is a sum over the contributions from all $F^{(\gamma)},\ \gamma\in \Gamma$. 
The contribution from $F^{(\varepsilon)}$ is ${2^{1-r}}(\mathcal{Q}({\mathcal{D}_*})(\I))^r$,
as before. Calculation of the contributions from  $F=F^{(\gamma)},\ \gamma\not=\varepsilon$ requires more work:
\begin{prop} \label{contributions} 
The contribution of the fixed point manifold $F=F^{(\gamma)},\ \gamma\not=\varepsilon$ to 
$\ca{Q}(M)(\I)$ is  
\[ \int_F \frac{\widehat{A}(F)\ \on{Ch}(\ca{L}_F,\, \I)^{1/2}}{D_\R(\nu_F,\I) }
=2^{1-r}\left(\tfrac{k}{2}+1 \right)^{\mathsf{l}(\gamma)/2}\varphi^{(\gamma)},\]
where the scalar $\varphi^{(\gamma)}=\mu_{F^{(\gamma)}}(\I)\in\on{U}(1)$ is the action of $\I$ on the pre-quantum line bundle over $F^{(\gamma)}$.
\end{prop}
\begin{proof}
Since $F=F^{(\gamma)}$ is a torus, $\widehat{A}(F)=1$. To compute the $D_\R$-class, note that 
the normal bundle of $T\J$ in $\D_*$ is an orientable real line bundle, hence it is trivializable.  
Consequently, the normal bundle $\nu_{\ti{F}^{(\gamma)}}$ to $\ti{F}^{(\gamma)}$ in $\ti{M}$ 
is trivializable, and thus the normal bundle $\nu_{F} = \nu_{\ti{F}^{(\gamma)}}/\Gamma$ to
$F$ in $M$ is a flat Euclidean vector bundle of rank $2r-\mathsf{l}(\gamma)$.  The element $\I$ acts by multiplication by $-1$ on the fibers of $\nu_{F} $, since 
$\Ad(\I)$ has order $2$ and cannot act trivially.  By definition of the $D_\R$-class 
(see \cite[Section 2.3]{al:fi} or \cite[Section 5.3]{me:twi}), it follows that 
$$
D_\R(\nu_{F}, \I) = i^{\on{rank}(\nu_F)/2} {\det}_\R^{1/2}(1-(-1))=
(2i)^{r-\tfrac{\mathsf{l}(\gamma)}{2}}.
$$

By  \cite[Proposition 9.3]{me:su}, the restriction $TM|_F$ inherits a distinguished $\on{Spin}_c$-structure (depending on the 
choice of level $k$ pre-quantization), equivariant for the action of $\I$.  
The line bundle $\mathcal{L}_{F} \to F$ is 
the $ \on{Spin}_c$-line bundle associated to this $ \on{Spin}_c$-structure, and 
\[\on{Ch}(\ca{L}_F,\, \I)^{1/2}=\sigma(\mathcal{L}_{F})(\I)^{1/2}\exp(\tfrac{1}{2} c_1(\mathcal{L}_{F}))\] is the square root of its equivariant Chern character, with $\sigma(\mathcal{L}_{F})(\I)\in\on{U}(1)$ the action of $\I$ the $\on{Spin}_c$-line 
bundle. As discussed in \cite[Section 2.3]{al:fi} (see also
\cite[Section 5.3]{me:twi}), the sign of the square root is determined
as follows. Since $\Phi$ restricts to a surjective map $F\to T$, the
fixed point set $F$ meets $\Phi^{-1}(e)$. Pick any $x\in F\cap
\Phi^{-1}(e)$. Observe that $\omega$ is non-degenerate at points of
$\Phi^{-1}(e)$, and choose a $\I$-invariant compatible complex
structure to view $T_xM$ as a Hermitian vector space. Let $A\in
\on{U}(T_xM)$ be the transformation defined by $\I$ and $A^{1/2}$ its
unique square root for which all eigenvalues are of the form
$e^{iu}$ with $0\le u<\pi$.  Then
\[ \sigma(\mathcal{L}_{F})(\I)^{1/2}=\varphi^{(\gamma)} {\det}_\C(A^{1/2}).\]
Since $\I$ acts trivially on $T_mF$ and as $-1$ on the normal bundle, the transformation $A^{1/2}$ acts trivially on  $T_xF$ and as $i$ on the normal bundle. Thus ${\det}_\C(A^{1/2})=i^{r-\mathsf{l}(\gamma)/2}$, which cancels a similar factor in the expression for the $D_\R$-class. 

It remains to find the integral $\int_{F}\exp(\tfrac{1}{2} c_1(\mathcal{L}_{F}))$. To this end, we 
interpret $\ca{L}_F$ as a pre-quantum line bundle. By the same argument as in Property (g) of Section \ref{subsec:preq}, (see also \cite[Section 11.1]{me:su}),  the level $k$ pre-quantization and the canonical twisted $\on{Spin}_c$-structure
on $M$ combine to give an element of $H^3(\Phi\,;\Z)$ at level $2k+4$. 
Since $H^2(M\,;\R)=0$, this element defines a pre-quantization at level $2k+4$. 
Pull-back to $F$ defines a level $2k+4$ pre-quantization of $F$, with $\ca{L}_F$ as the
pre-quantum line bundle. Hence $c_1(\mathcal{L}_F)$ is 
the $2k+4$-th multiple of the class of the symplectic form on $F$. It follows 
that 
\[ \int_{F}\exp(\tfrac{1}{2} c_1(\mathcal{L}_{F}))=(k+2)^{\mathsf{l}(\gamma)}\on{vol}(F)=2^{1-\mathsf{l}(\gamma)/2} \left(\tfrac{k}{2}+1 \right)^{\mathsf{l}(\gamma)/2}\]
where we have used Lemma \ref{restofform}. 
\end{proof}

The phase factors $\varphi^{(\gamma)}$ depend on the choice of pre-quantization. 
Recall again that the set of pre-quantizations of a quasi-Hamiltonian $\SU(2)$-space is a torsor under the 
group of isomorphism classes of flat line bundles. In our case this is the group 
\[ \on{Tor}(H^2(M\,;\Z))\cong \on{Hom}(\Gamma,\on{U}(1)).\] 
The homomorphism $\psi\colon \Gamma \to \on{U}(1)$ defines the flat line bundle 
$\ti{M}\times_\Gamma \C_\psi$, where $\C_\psi$ is the 1-dimensional $\Gamma$-representation 
defined by $\psi$.  Changing the pre-quantization by such a flat line bundle  changes $\varphi^{(\gamma)}$ for 
$F=F^{(\gamma)}$ to  $\psi(\gamma)\varphi^{(\gamma)}$. By Property (g) of Section \ref{subsec:preq}, 
and since $H^2(M;\R)=0$, there is a \emph{distinguished} pre-quantization at any level $k\in 4\N$. Hence, the inequivalent pre-quantizations at level $k\in 4\N$ are labeled by 
$\on{Hom}(\Gamma,\on{U}(1))$. 

\begin{lemma} If $r\ge 3$ and $k\in 4\N$, the phase factor for the pre-quantization labeled by $\psi\in \on{Hom}(\Gamma,\on{U}(1))$ is given by 
\[ \varphi^{(\gamma)} = (-1)^{\tfrac{k}{4}(r-\mathsf{l}(\gamma)/2)}\psi(\gamma).\]

\end{lemma}
\begin{proof} 
The phase factor $\varphi^{(\gamma)}$ for the distinguished pre-quantization at level $4$ is given 
by ${\det}_\C(A)=(-1)^{r-\mathsf{l}(\gamma)/2}$, in the notation from the proof of Proposition 
\ref{contributions}. For the distinguished pre-quantization at level 
$k\in 4\N$, we have to take the $\tfrac{k}{4}$-th power of this number, and changing the pre-quantization by   $\psi$ we have to multiply by $\psi(\gamma)$. 
\end{proof}

If $r=2$, there are $|\Gamma|=2$ distinct pre-quantizations at all even levels $k\in 2\N$, 
related by elements $\psi\in \on{Hom}(\Gamma,\on{U}(1))$. 
Aside from the discrete fixed point set $F^{(\varepsilon)}$, there is a single non-discrete fixed point component $F^{(\gamma)}$ of $\I$, given by $\gamma=(c,c)$. The non-trivial homomorphism 
$\psi\in \on{Hom}(\Gamma,\on{U}(1))\cong\Z_2$ satisfies $\psi(c,c)=-1$, 
hence the weight $\varphi^{(\gamma)}$ 
is equal to $1$ for one of the pre-quantizations and $-1$ for the other. 

\subsection{Quantization of $M$}\label{subsec:quant}

We are now ready to summarize our computation of $\ca{Q}(M)$ for $M=(\D_*\times\cdots \times \D_*)/\Gamma$. Assuming that $k$ is even, recall that $\D_*$ has a unique pre-quantization at level $k$, and 
$\ca{Q}(\D_*)=\tau_{k/2}$. Define an element 
$$
\chi = \tau_0 - \tau_2 + \tau_4 - \cdots + (-1)^{k/2} \tau_k\in R_k(\SU(2)). 
$$
By the orthogonality relations for $R_k(\SU(2))$, 
this element satisfies $\chi(\I) = (\tfrac{k}{2}+1)$ and $\chi(t)=0$ for $t=t_l$, 
$l\neq k/2$. 
Hence we may write the sum over the fixed point contributions as follows: 
\[ \ca{Q}(M)(t)=2^{1-r} \Big(\tau_{k/2}(t)^r+\chi(t) \sum_{\gamma \in \Gamma \setminus \{\varepsilon \}}
\left( \tfrac{k}{2} +1\right)^{\mathsf{l}(\gamma)/2-1} \varphi^{(\gamma)} \Big)\]

\begin{thm} \label{th:main}
Consider the quasi-Hamiltonian $\SU(2)$-space $M=(\D_*\times\cdots
\times\D_*)/\Gamma$ with $r\ge 2$ factors, where $\Gamma\subset Z^r$
consists of all $\gamma=(\gamma_1,\ldots,\gamma_r)$ with
$\prod_{i=1}^r \gamma_i=e$.
\begin{enumerate}
\item
If $r\ge 3$, the space $M$ is pre-quantized at level $k$ if and only if $k\in 4\N$. The different pre-quantizations 
are indexed by the elements $\psi\in \on{Hom}(\Gamma,\on{U}(1))$, and the corresponding level $k$ quantization is 
given by the formula,
$$
\mathcal{Q}_\psi (M) = 2^{1-r} \Big((\tau_{k/2})^r + \chi \sum_{\gamma \in \Gamma \setminus \{\varepsilon \}}   \psi(\gamma) ( \tfrac{k}{2} +1)^{\f{\mathsf{l}(\gamma)}{2}-1} (-1)^{\tfrac{k}{4} (r-\f{\mathsf{l}(\gamma)}{2})} \Big).
$$
\item
If $r=2$, the space $M$ is pre-quantized at level $k$ if and only if $k\in 2\N$. At any such level, there
are two distinct pre-quantizations indexed by the action $\pm 1$ of $\I$ on the pre-quantum line bundle over $F^{(\gamma)}$, for $\gamma=(c,c)$. 
The corresponding level $k$ quantizations of $M$ are 
$$
\mathcal{Q}_\pm(M) = \frac{1}{2}\left((\tau_{k/2})^2 \pm  \chi \right).
$$
\end{enumerate}
\end{thm}

\subsection{Multiplicity computations}
Being elements of $R_k(\SU(2))$, the coefficients of $\ca{Q}(M)$ in its decomposition with respect to the basis 
$\tau_0,\ldots,\tau_k$ must be integers. In this Section, we will compute these multiplicities for small $r$. 
\subsubsection{$r=2$ factors}
Assume $k\in 2\N$, and let $\ca{Q}_\pm(M)$ be the quantizations corresponding to the pre-quantizations labeled by $\pm 1$. The multiplication rules for level $k$ characters give 
\[ (\tau_{k/2})^2=\tau_0+\tau_2+\ldots+\tau_k.\]
Hence, if $k\in 4\N$ we obtain 
\[\begin{split} 
\ca{Q}_+(M)&=\tau_0+\tau_4+\ldots+\tau_{k},\\
\ca{Q}_-(M)&=\tau_2+\tau_6+\ldots+\tau_{k-2},\\
\end{split}\] 
while for $k\in 4\N-2$,
\[\begin{split} 
\ca{Q}_+(M)&=\tau_0+\tau_4+\ldots+\tau_{k-2},\\
\ca{Q}_-(M)&=\tau_2+\tau_6+\ldots+\tau_{k}.\\
\end{split}\] 
 
\subsubsection{$r=3$ factors}
Let $\ca{Q}_\psi(M)$ denote the level $k\in 4\N$ pre-quantization indexed by 
$\psi\in \on{Hom}(\Gamma,\on{U}(1))$. 
Since $r=3$, $\mathsf{l}(\gamma)=2$ for any $\gamma \neq \varepsilon$ and the quantization formula simplifies to:
$$
\mathcal{Q}_\psi(M) = \frac{1}{4} \Big( \tau_{2m}^3 + \chi\,\sum_{\gamma \neq\varepsilon}  \psi(\gamma)\Big).
$$
For the trivial homomorphism $\psi=1$,  we have $\sum_{\gamma\neq \varepsilon} \psi(\gamma) =3$, while for a non-trivial homomorphism $\psi\neq 1$, $\sum_{\gamma\neq \varepsilon} \psi(\gamma) =-1$. 
We have, 
\[ (\tau_{k/2})^3=\tau_0+3\tau_2+\ldots+(\tfrac{k}{2}+1)\tau_{k/2}
+\ldots+3\tau_{k-2}+\tau_k.\]
We therefore obtain 
\[ \begin{split}
\mathcal{Q}_\psi(M)&=(\tau_0+2\tau_4+3\tau_8+
\ldots 3\tau_{k-8}+2\tau_{k-4}+\tau_k)\\
&\ \ \ +(\tau_6+2\tau_{10}+\ldots+2\tau_{k-10}+\tau_{k-6})\ \ \ \ \ \ \ \ \ \ \ \ \ \ \ \ \ \ \ \ \ \ \ \mbox{ if $\psi=1$},\\
\mathcal{Q}_\psi(M)&=(\tau_0+\tau_4+2\tau_8+\ldots+3\tau_{k-8}+2\tau_{k-4}+\tau_k)\\
&\ \ \ +(\tau_2+2\tau_6+3\tau_{10}+\ldots+3\tau_{k-10}+2\tau_{k-6}+\tau_{k-2})
\ \ \ \ \mbox{ if $\psi\neq 1$}.
\end{split}\]
Note that the coefficients are symmetric about the midpoint $\tfrac{k}{2}$ of the interval $[0,k]$. 
In closed form, $\ca{Q}_\psi(M)=\sum_{j=0}^{k/2} a_{2j} \tau_{2j}$, where 
\[ a_{2j}=\begin{cases} \tfrac{1}{4} (2j+1+ 
(4\delta_{\psi,1}-1)(-1)^j)& \colon \ 2j\le k/2,\\
\tfrac{1}{4} (k-2j+1+(4 \delta_{\psi,1}-1)(-1)^j)
& \colon \ 2j\ge k/2.
\end{cases}\]

\subsubsection{$r=4$ factors} 
If $r=4$ we have $|\Gamma|=8$. There is a unique element $\gamma'\in \Gamma$ with $\mathsf{l}(\gamma')=4$, and 
$\mathsf{l}(\gamma)=2$ for $\gamma\neq \gamma',\varepsilon$. Hence 
we may write the quantization formula for levels $k\in 4\N$ as:
$$
\mathcal{Q}_\psi(M) = \frac{1}{8} \Big(\tau_{k/2}^4 + \big(   \psi(\gamma')\left(\tfrac{k}{2}+1 \right) 
+(-1)^{k/4} 
\sum_{\mathsf{l}(\gamma)=2} \psi(\gamma)
\big) \chi \Big).
$$
One finds that there are $4$ homomorphisms $\psi$ with $\sum_{\mathsf{l}(\gamma)=2} \psi(\gamma)=0,\ 
\psi(\gamma')=-1$ and $3$ homomorphisms with $\sum_{\mathsf{l}(\gamma)=2} \psi(\gamma)=-2,\ 
\psi(\gamma')=1$. Of course, $\sum_{\mathsf{l}(\gamma)=2} \psi(\gamma)=6,\ 
\psi(\gamma')=1$ for $\psi=1$.  
Therefore,  we have 
$$
\mathcal{Q}_\psi(M) = 
\begin{cases}
\frac{1}{8} \left(\tau_{k/2}^4 + \left(6 (-1)^{k/4} 
 +  \left(\tfrac{k}{2}+1 \right) \right) \chi \right)& \colon\ \psi=1\\
 \frac{1}{8} \left(\tau_{k/2}^4 -  
  \left(\tfrac{k}{2}+1 \right) \chi \right)& \colon\ \sum_{\mathsf{l}(\gamma)=2} \psi(\gamma)=0\\
 \frac{1}{8} \left(\tau_{k/2}^4 + \left(2 (-1)^{k/4+1} 
 +  \left(\tfrac{k}{2}+1 \right) \right) \chi \right)&\colon\ \sum_{\mathsf{l}(\gamma)=2} \psi(\gamma)=-2
 \end{cases}
 $$
with 
$$
(\tau_{k/2})^4 = \sum_{j=0}^{k/2} (\tfrac{k}{2}+1-2j^2+jk) \tau_{2j}.
$$
One may verify that the multiplicities of $\tau_{2j}$ in $\ca{Q}_\psi(M)$ are integers, 
as required.

\section{Fuchs-Schweigert} \label{sec:Smatrix}

The formulas appearing in Theorem \ref{th:main} may be rewritten in terms of the so-called $S$-matrix.  
For $z \in Z$, define $S^{(z)}_{m,l}$ by 
\[ S^{(z)}_{m,l}=\begin{cases}1 &\mbox{ if } z=c\\
\mbox S_{m,l}&\mbox{ if } z=e.
\end{cases}\]
In the terminology of \cite{fu:ac}, $S^{(z)}_{m,l}$ is the $S$-matrix of the \emph{orbit Lie algebra} associated to the central element $z$. (This interpretation may seem obscure for $\SU(2)$, 
but becomes natural for higher rank groups.) Consider once again the space $M=\ti{M}/\Gamma$ from 
\eqref{eq:M}. Recall that $\Gamma$ consists of elements
$\gamma=(\gamma_1,\ldots,\gamma_{s+2h})\in Z^{s+2h}$ such that 
$\prod_{j=1}^s \gamma_j=e$, and $\gamma_j=e$ for all $j\le s$ with $\Co_j\neq \Co_*$. 
In particular $|\Gamma|=2^{2h+r-1}$ if $r\ge 1$, while $|\Gamma|=2^{2h}$ if $r=0$. 
To write the Fuchs-Schweigert formula, it is convenient to use the following notation. 
For $\gamma\in\Gamma$, let  $\sum_l^{(\gamma)}$ denote the full sum $\sum_{l=0}^k$ if all
$\gamma_i=e$, and consisting of the single term $l=\f{k}{2}$ if at least one 
$\gamma_i\neq e$. (For higher rank groups, this becomes a sum over level $k$ weights that are fixed
under the action of all $\gamma_i\in Z$ on the set of level $k$
weights.) We will prove the following equivariant analogue to the
Fuchs-Schweigert formula:
\begin{thm} \label{th:FSconjecture}
Suppose the quasi-Hamiltonian $\SU(2)$-space
\[ M=\big(\D_1\times\cdots\times\D_s\times \mathbf{D}(\SU(2))^h\big)/\Gamma\]
is pre-quantized at level $k$. Then 
\begin{equation}\label{eq:fs}
\mathcal{Q}(M) = \frac{1}{|\Gamma|} \sum_{\gamma \in \Gamma} \varphi'(\gamma) 
{\sum_l}^{(\gamma)}\  \frac{S^{(\gamma_1)}_{m_1,l}\cdots S^{(\gamma_s)}_{m_s,l}}{(S_{0,l})^{s+2h}}\ \ti{\tau}_l ,
\end{equation}
where $\varphi'(\gamma)\in\on{U}(1)$ are phase factors depending on the choice of pre-quantization, 
with $\varphi'(\varepsilon)=1$.  
\end{thm}
An explicit description of the phase factors $\varphi'(\gamma)$ will
be given during the course of the proof.  
\begin{proof}[Proof of Theorem \ref{th:FSconjecture}]
The space $M$ is a fusion product of the space $\wt{(\Co_*)^r}$, conjugacy classes $\D_j\neq \D_*$, 
and $h$ factors of $\mathbf{D}(\SO(3))$ (viewed as a quasi-Hamiltonian $\SU(2)$-space). Since the 
fusion product in the basis $\ti{\tau}_m$ is diagonalized, we may verify the formula separately for factors of these three types.  

We begin with the case $h=0,s=r$, with $r\ge 3$ (thus necessarily $k\in 4\N$). We re-write the right hand side of \eqref{eq:fs}, 
separating the term $\gamma=\varepsilon$ from the sum over terms $\gamma\neq\varepsilon$.  
The right hand side of \eqref{eq:fs} becomes 
\begin{equation}\label{eq:fs1}
 \ca{Q}(M)=\frac{1}{|\Gamma|} \Big(\varphi'(\varepsilon)\sum_{l}
\frac{(S_{k/2,l})^r}{(S_{0,l})^{r}}\ti{\tau}_l 
+\sum_{\gamma\neq \varepsilon}\varphi'(\gamma)  \frac{(S_{k/2,k/2})^{r-\mathsf{l}(\gamma)}}{(S_{0,k/2})^{r}}\ \ti{\tau}_{k/2}\Big).
\end{equation}
The sum over $l$ is just $(\tau_{k/2})^r$. 
The element $\chi\in R_k(\SU(2))$ considered in Section 
\ref{subsec:quant} satisfies $\chi(t_l)=(\frac{k}{2}+1)\delta_{l,k/2}$ for 
$l=0,\ldots,k$,
hence
\[ \ti{\tau}_{\frac{k}{2}}=(\tfrac{k}{2}+1)^{-1}\chi .\]
Furthermore, by definition of the $S$-matrix, 
\[ S_{0,k/2}=(\tfrac{k}{2}+1)^{-\f{1}{2}},\ \ \ \ S_{k/2,k/2}=(\tfrac{k}{2}+1)^{-\f{1}{2}}(-1)^{\f{k}{4}}.\] 
Equation \eqref{eq:fs1} becomes
\[ \ca{Q}(M)=\frac{1}{2^{r-1}} \Big(\varphi'(\varepsilon) (\tau_{k/2})^r+
\sum_{\gamma\neq \varepsilon}\varphi'(\gamma) 
(-1)^{\f{k}{4}(r-\mathsf{l}(\gamma))}
(\tfrac{k}{2}+1)^{\f{\mathsf{l}(\gamma)}{2}-1}\chi\Big)\]
which agrees with Theorem \ref{th:main} for $\varphi'(\gamma)=\psi(\gamma) (-1)^{\f{k\,\mathsf{l}(\gamma)}{8}}$. 

The calculation is similar for the case $h=0, s=r=2,\ k\in 2\N$.  Here, $|\Gamma|=2$, and the generator $\gamma=(c,c)\in \Gamma$ has $\mathsf{l}(\gamma)=2$. We hence obtain 
\[ \ca{Q}(M)=\frac{1}{2} \Big(\varphi'(e,e) (\tau_{k/2})^r+
\varphi'(c,c) \chi\Big)
\]
which agrees with  Theorem \ref{th:main} if we put $\varphi'(e,e)=1$, and $\varphi'(c,c)=\pm 1$. 
If $h=0$ and $s=r=1,\ k\in 2\N$, then $\Gamma=\{e\}$, and the formula becomes 
$\ca{Q}(M)=\varphi'(e)\tau_{k/2}$, which is the correct expression for $\ca{Q}(\D_*)$ for 
$\varphi'(e)=1$. Similarly, if $h=r=0,\ s=1$ so that $M$ is a conjugacy class $\D_j\neq\D_*$, the 
formula reduces to $\ca{Q}(M)=\tau_{m_j}=\ca{Q}(\D_j)$. 
 
Consider finally the case $h=1,\,s=0$ so that $M=\mathbf{D}(\SO(3))$. Pre-quantizability of this space requires $k\in 2\N$, and as shown in \cite{me:su} the distinct pre-quantizations are indexed by 
$\varphi\in \on{Hom}(\Gamma,\on{U}(1))$, with  
$\Gamma=Z\times Z$. Separating off the term $(e,e)$, \eqref{eq:fs} becomes
\[ \ca{Q}(M)=\f{1}{4}\Big(\varphi'(\epsilon)\sum_{l}\f{1}{S_{0,l}^2}
\ti{\tau}_l+\sum_{\gamma\neq (e,e)}\varphi'(\gamma)\f{1}{S_{0,k/2}^2}\ti{\tau}_{k/2}\Big).\]
We have $\f{1}{S_{0,k/2}^2}\ti{\tau}_{k/2}=\chi$, and 
\[ \ca{Q}(\mathbf{D}(\SU(2)))=\sum_m \tau_m^2=\sum_{l,m} \f{S_{m,l}^2}{S_{0,l}^2}
\ti{\tau}_l=\sum_l \f{1}{S_{0,l}^2}\ti{\tau}_l,\]
where we use the symmetry and orthogonality of the $S$-matrix. Thus the formula may be re-written 
\[ \ca{Q}(M)=\f{1}{4}\Big(\varphi'(\epsilon)\ca{Q}(\mathbf{D}(\SU(2)))+\sum_{\gamma\neq (e,e)}\varphi'(\gamma)\ \chi\Big).\]
This agrees with the formula for $\ca{Q}(\mathbf{D}(\SO(3))$ given in \cite[Section 11.4]{me:su} if one puts 
$\varphi'(e,e)=1$ and $\varphi'(\gamma)=(-1)^{k/2}\varphi(\gamma)$ for $\gamma\neq(e,e)$. 
\end{proof}

By combining this result
with the `quantization commutes with reduction' theorem for
quasi-Hamiltonian spaces \cite[Theorem 10.1]{me:su}, and since the
coefficient of $\tau_0$ in $\ti{\tau}_l$ is $S_{0,l}^{2}$, we obtain
the Fuchs-Schweigert formula \cite{fu:ac} for the $\SO(3)$ moduli space
$\M(\Sigma,\Co_1,\ldots,\Co_s)$, where $\Sigma$ is of genus $h$ with $s$ boundary components. Recall that this moduli space has up two $2$ 
connected components, of the form $M/\!\!/\SU(2)$ for suitable choice of lifts $\D_j$. 
We have, 
\begin{equation}\label{eq:noneqFS}
 \ca{Q}(M/\!\! /\SU(2))=\frac{1}{|\Gamma|} \sum_{\gamma \in \Gamma} \varphi'(\gamma) {\sum_{l}}^{(\gamma)} \frac{S^{(\gamma_1)}_{m_1,l}\cdots S^{(\gamma_s)}_{m_s,l}}{(S_{0,l})^{s+2h-2}}.
 \end{equation}
\begin{remark}
The above Fuchs-Schweigert type formula computes the quantization of the moduli space of $\SO(3)$-bundles interpreted as the \emph{index} of a pre-quantum line bundle, while the original conjecture in \cite{fu:ac} concerns the dimension of the space of conformal blocks. It is expected that, just as in the case of simply-connected groups, the space of conformal blocks can be re-interpreted as the space of holomorphic sections, and that a Kodaira vanishing result can further identify its dimension with the index considered here. We are not aware of a reference addressing 
such questions in generality for non-simply connected groups.    
\end{remark}

\def\cprime{$'$} \def\polhk#1{\setbox0=\hbox{#1}{\ooalign{\hidewidth
  \lower1.5ex\hbox{`}\hidewidth\crcr\unhbox0}}} \def\cprime{$'$}
  \def\cprime{$'$} \def\cprime{$'$}
  \def\polhk#1{\setbox0=\hbox{#1}{\ooalign{\hidewidth
  \lower1.5ex\hbox{`}\hidewidth\crcr\unhbox0}}} \def\cprime{$'$}
  \def\cprime{$'$} \def\cprime{$'$} \def\cprime{$'$}
\providecommand{\bysame}{\leavevmode\hbox to3em{\hrulefill}\thinspace}
\providecommand{\MR}{\relax\ifhmode\unskip\space\fi MR }
\providecommand{\MRhref}[2]{%
  \href{http://www.ams.org/mathscinet-getitem?mr=#1}{#2}
}
\providecommand{\href}[2]{#2}

\end{document}